\documentclass[reqno]{amsart}
\usepackage[T1]{fontenc}
\usepackage{amsmath, amssymb, amsthm, amsfonts}
\usepackage{mathrsfs}

\usepackage{lmodern}
\usepackage{fontawesome5}
\usepackage{booktabs}
\usepackage{caption}
\usepackage{float}

\usepackage{graphicx}
\usepackage[dvipsnames]{xcolor}
\usepackage{mdframed}
\usepackage{tikz-cd}

\usepackage{fancyhdr}
\usepackage[hang,flushmargin]{footmisc}

\usepackage{enumitem}
\usepackage{mathtools}

\usepackage[alphabetic, initials, msc-links]{amsrefs}
\makeatletter
\renewcommand{\BibLabel}{%
    \Hy@raisedlink{\hyper@anchorstart{cite.\CurrentBib}\hyper@anchorend}%
    [\thebib]%
}
\makeatother
\makeatletter
    \renewcommand{\MR}[1]{\ (\href{https://mathscinet.ams.org/mathscinet-getitem?mr=MR#1}{MR#1})}
\makeatother

\usepackage{thmtools}

\usepackage{hyperref}
\hypersetup{
    colorlinks=true,
    linkcolor=-Cerulean,
    filecolor=magenta,
    urlcolor=Cerulean,
    citecolor=Cerulean
}

\usepackage[capitalize, nameinlink]{cleveref}
\crefname{equation}{}{}
\crefname{prop}{Proposition}{Propositions}
\crefname{section}{\S}{}

\graphicspath{ {./Images/} }

\theoremstyle{definition}
\newtheorem{theorem}{Theorem}[section]

\newtheorem{definition}[theorem]{Definition}

\newtheorem*{remark}{Remark}

\DeclareMathOperator{\GL}{\textsf{GL}}
\DeclareMathOperator{\SL}{\textsf{SL}}
\DeclareMathOperator{\im}{\textsf{im}}

\title[Cuspidal Cohomology Computations]{Computations Directly on the Cuspidal Cohomology of Congruence Subgroups of $\mathrm{SL}(3, \mathbb{Z})$}
\author{Zachary Porat}
\address{Department of Mathematics and Computer Science, Wesleyan University, 265 Church Street, Middletown, Connecticut 06459}
\email{\href{mailto:zporat@wesleyan.edu}{\tt zporat@wesleyan.edu}}
\urladdr{\href{https://zporat.github.io}{\tt https://zporat.github.io}}


\begin{document}

\begin{abstract}
    Ash, Grayson, and Green [J.\ Number Theory 19 (1984), pp.\ 412--436] compute the action of Hecke operators on a certain subspace of the cohomology of low-level congruence subgroups of $\mathsf{SL}(3, \mathbb{Z})$.  This subspace contains the cuspidal cohomology, which is of primary interest.  We extend their work, introducing a method that allows for computing the action of Hecke operators directly on the cuspidal cohomology.  Using this method, we obtain data for prime level less than 3500, finding seven additional levels at which nonzero cuspidal classes appear and calculating local factors for five of these levels.
\end{abstract}

\maketitle

\section{Introduction} \label{section:intro}

Let $G$ be a reductive Lie group, $X$ the associated symmetric space for $G$, and $\Gamma$ a discrete subgroup of $G$.  Automorphic forms arise in the cohomology of certain compactifications of $X/\Gamma$.  For example, consider the classical setting where $G = \GL(2)$, $X$ is the upper half plane, and $\Gamma = \Gamma_0(2, N)$.  One can study the space of weight two cusp forms $S_2(\Gamma_0(2, N))$ by instead investigating specific classes in the cohomology of the modular curve $X_0(N)$.

Ash, Grayson, and Green \cite{AGG84} build this theory for $G = \GL(3)$.  In this case, $X = \mathsf{SO}(3) \backslash \mathsf{SL}(3, \mathbb{R})$, where $\mathsf{SL}(3, \mathbb{R})$ acts on the right.  They consider the congruence subgroup
\[\Gamma_0(3, N) = \{ (a_{ij}) \in \SL(3, \mathbb{Z}) \colon a_{i1} \equiv 0 \pmod N \text{ for } i =2, 3\}. \]
Of particular interest are cuspidal automorphic forms.  Similar to the classical setting, a cuspidal automorphic form $f$ on $\GL(3)$ corresponds to a nonzero cuspidal cohomology class $u_f$ in the cohomology of the compactification of $X/\Gamma_0(3, N)$.  A complete discussion of this correspondence can be found in \cite{LS82}.

Let $u_f$ be an eigenclass for all Hecke operators $T_{A(\ell)}$ with complex eigenvalues $a_\ell$ for prime $\ell$.  In the $\GL(3)$ setting,
\[
    L_\ell(f, s) = (1 - a_\ell \ell^{-s} + \overline{a_\ell} \ell^{1-2s} - \ell^{3 - 3s})^{-1}
\]
is the local factor of the associated $L$-function for $\ell \neq p$, when $\Gamma_0(3, p)$ is the level of the form.  Alternatively, one can view this data through the lens of Galois representations, where $a_\ell$ is the trace of the Frobenius at $\ell$ for a certain 3-dimensional $\textsf{Gal}(\overline{\mathbb{Q}}/\mathbb{Q})$ representation and $p$ is the conductor of said representation.  We do not pursue this avenue, but discussions can be found in \cite{AM92} and \cite{vGT94}.

For prime level $p < 100$, Ash, Grayson, and Green \cite{AGG84} provide computations that can be used to calculate several of these local factors. In \cite{vGvdKTV97}, van Geemen, van der Kallen, Top, and Verberkmoes extend these computations, checking composite level $N < 260$ and prime level $p \leq 337$.  This paper returns to the prime level case, reporting additional computations of local factors for prime level $p < 2000$ (see \cref{table:2}).

We start by recalling the necessary background information from \cite{AGG84}.  Then, we present an alternate strategy for understanding the action of the Hecke algebra on the cuspidal cohomology.  In contrast to previous approaches, our strategy allows for computation of the Hecke operators directly on the cuspidal cohomology, as opposed to other larger spaces.

We conclude with our computational methods and results, detailing the advantages of this new strategy.  For prime level $p < 3500$, we found nonzero cuspidal classes at levels 521, 953, 1289, 1433, 1913, 2089, and 2833.  We were able to calculate eigenvalues for desired Hecke operators at levels 521, 953, 1289, 1433, and 1913.  However, computational constraints prevented us from finding the eigenvalues at levels 2089 and 2833.

\section{Preliminaries} \label{section:prelims}

\subsection{Vector Space Interpretation} \label{sec:vectorspace}
Let $G = \GL(3)$ and consider the symmetric space $X = \mathsf{SO}(3) \backslash \mathsf{SL}(3, \mathbb{R})$, where $\mathsf{SL}(3, \mathbb{R})$ acts on the right.  Henceforth, we let $\Gamma = \Gamma_0(3, p)$, where
\[\Gamma_0(3, p) = \{ (a_{ij}) \in \SL(3, \mathbb{Z}) \colon a_{i1} \equiv 0 \pmod p \text{ for } i =2, 3\}. \]
Note, we are restricting our attention to congruence subgroups of prime level $p$.

By work of Lee and Schwermer \cite{LS82}, there is a correspondence between cuspidal automorphic forms on $\GL(3)$ and nonzero cuspidal cohomology classes in $H^3(\mathcal{M}, \mathbb{C})$, where $\mathcal{M}$ denotes the Borel-Serre compactification of $X/\Gamma$.  We leverage this correspondence to study the space of cuspidal automorphic forms on $\textsf{GL}(3)$ by instead investigating a specific space of cohomology classes called the \textbf{cuspidal cohomology}
\[H^3_\text{cusp}(\Gamma, \mathbb{C}) = \{u_f \in H^3(\mathcal{M}, \mathbb{C}) \colon u_f \text{ with } f \text{ a cuspidal automorphic form}\}.\]
We note that the cohomology and homology of $\Gamma$ and $\mathcal{M}$ with complex coefficients are canonically isomorphic and can be identified, as done in \cite{AGG84}.

Ash, Grayson, and Green \cite{AGG84}*{Thm.\ 3.2, Prop.\ 3.12} establish that the homology $H_3(\Gamma, \mathbb{C})$, the dual space of $H^3(\Gamma, \mathbb{C})$, is isomorphic to a certain vector space of $\mathbb{C}$-valued functions on $\mathbb{P}^2(\mathbb{Z}/p\mathbb{Z})$.
\begin{theorem}[\cite{AGG84}*{Thm.\ 3.2, Prop.\ 3.12}] \label{thm:W}
    There exists an isomorphism of vector spaces $\Phi \colon W \to H_3(\Gamma, \mathbb{C})$, where $W$ denotes the vector space of functions $f \colon \mathbb{P}^2(\mathbb{Z}/p\mathbb{Z}) \to \mathbb{C}$ that satisfy the following properties:
\begin{enumerate}[label=(\roman*)]
    \item $f(x \colon y \colon z) = f(z \colon x \colon y) = f(-x \colon y \colon z) = -f(y \colon x \colon z)$;
    \item $f(x \colon y \colon z) + f(-y \colon x -y \colon z) + f(y -x \colon -x \colon z) = 0$.
\end{enumerate}
\end{theorem}

\noindent Moreover, they show that the homology can be understood via an appropriate vector space decomposition into two subspaces.  The first subspace is related to the cuspidal cohomology $H^3_{\text{cusp}}(\Gamma, \mathbb{C})$ and the second is related to the space of weight two cusp forms $S_2(\Gamma_0(2, p))$.  To better understand each component and how they interact, they make the following definitions.

\begin{definition}[\cite{AGG84}*{Def.\ 3.10, Def.\ 3.13}] \label{def:1}
    Let $\Delta(p)$ be the subgroup of $\GL(2, \mathbb{Z})$ generated by $\Gamma_0(2, p)$ and the matrix
    \[s = \begin{pmatrix}
        -1 & 0 \\
        0 & 1
    \end{pmatrix}.\]
    Let $W_0(\Delta(p))$ denote the vector space of functions $f \colon \mathbb{P}^1(\mathbb{Z}/p\mathbb{Z}) \to \mathbb{C}$ that satisfy
    \begin{enumerate}[label=(\roman*)]
        \item $f(x \colon y) = f(-x \colon y) = -f(y \colon x)$;
        \item $f(x \colon y) + f(-y \colon x-y) + f(y-x \colon -x) = 0$;
        \item $f(1 \colon 0) = 0$.
    \end{enumerate}
\end{definition}

\begin{definition}[\cite{AGG84}*{Def.\ 3.15}] \label{def:2}
    Let $\alpha, \beta \colon W_0(\Delta(p)) \to W$ be the linear maps
    \[(\alpha f)(x \colon y \colon z) = f(x \colon y) + f(y \colon z) + f(z \colon x),\]
    and
    \[(\beta f)(x \colon y \colon z) =
    \begin{cases}
        0 & \text{ if } xyz \neq 0 \\
        f(x \colon y) & \text{ if } z = 0 \\
        f(y \colon z) & \text{ if } x = 0 \\
        f(z \colon x) & \text{ if } y = 0.
    \end{cases}
    \]
    Set $f(0 \colon 0) = 0$.
\end{definition}

\begin{definition}[\cite{AGG84}*{Def.\ 3.16}] \label{def:3}
Let $A, B \colon W \to W_0(\Delta(p))$ be linear maps such that
\begin{align*}
    (Ag)(x \colon y) & = \sum_{z \in (\mathbb{Z}/p\mathbb{Z})^*} g(x \colon y \colon z); \\
    (Bg)(x \colon y) & = g(x \colon y \colon 0).
\end{align*}

\end{definition}

\begin{theorem}[\cite{AGG84}*{Thm.\ 3.19}] \label{thm:Wnc}
    Let $\delta \colon W_0(\Delta(p))^2 \to W$ be the linear map $\delta(f_1, f_2) = \alpha(f_1) + \beta(f_2)$.  Then, $\delta$ is injective, and $\textsf{im}(\delta) = W^{\text{nc}}$.
\end{theorem}

In the preceding theorem, $W^\text{nc}$ denotes the preimage $\Phi^{-1}(H_3^{\text{nc}}(\Gamma, \mathbb{C}))$, where $H_3^{\text{nc}}(\Gamma, \mathbb{C})$ is the complement of the algebraic dual of $H^3_{\text{cusp}}(\Gamma, \mathbb{C})$ described in \cite{AGG84}*{Def.\ 3.4, Lem.\ 3.5}.  Finally, we present \cite{AGG84}*{Sum.\ 3.23}, which shows that the cuspidal cohomology, the space of interest, can be viewed as follows:

\begin{theorem}[\cite{AGG84}*{Sum.\ 3.23}]\label{thm:summary}
    The cuspidal cohomology $H^3_{\text{cusp}}(\Gamma, \mathbb{C})$ is isomorphic to the subspace $U \subseteq W$ consisting of all $f \in U$ satisfying the following properties:
\begin{enumerate}[label=(\roman*)]
    \item $f(x \colon y \colon z) = f(z \colon x \colon y) = f(-x \colon y \colon z) = -f(y \colon x \colon z)$;
    \item $f(x \colon y \colon z) + f(-y \colon x -y \colon z) + f(y -x \colon -x \colon z) = 0$;
    \item $f(x \colon y \colon 0) = 0$;
    \item $\sum_{z \in \mathbb{Z}/p\mathbb{Z}} f(x \colon y \colon z) = 0$.
\end{enumerate}
\end{theorem}

\noindent The results detailed in this section will provide a framework for understanding the vector space decomposition of $W$ in \cref{section:results}.

\subsection{Modular Symbols} \label{sec:mod_symbols}
We now briefly review the notion of modular symbols in the $\textsf{GL}(3)$ setting.  Let $Q$ be $3 \times 3$ rational matrix with nonzero rows.  The modular symbol $[Q]$ is an element of $H_1(T_3, \mathbb{Z})$, where $T_3$ is the Tits building for $\textsf{SL}(3, \mathbb{Q})$.  If $Q \in \textsf{SL}(3, \mathbb{Z})$, we say the modular symbol $[Q]$ is \textbf{unimodular}.

Modular symbols can be viewed more concretely as a collection of nonzero rational row vectors that enjoy the properties of \cite{AR79}*{Def.\ 2.2}.  Note however that we swap column vectors for row vectors and left action for right action.  These properties allow us to write any modular symbol as the finite sum of unimodular symbols.  Moreover, Ash and Rudolph \cite{AR79}*{Def.\ 3.1} show how we can view $[Q]$ as an element in $H_2(\mathcal{M}, \partial \mathcal{M} ; \mathbb{C})\simeq H^3(\Gamma, \mathbb{C})$.  This leads to \cref{thm:2.7}.

\begin{theorem}[\cite{AGG84}*{Prop.\ 3.24}] \label{thm:2.7}
    The intersection pairing
    \[\langle -, - \rangle \colon H_2(\mathcal{M}, \partial \mathcal{M}) \times H_3(\mathcal{M}) \to \mathbb{C}\]
    evaluates as
    \[\langle [Q], \Phi(f) \rangle = f(Q)\]
    for any $Q$ in $\textsf{SL}(3, \mathbb{Z})$ and $f \in W$.  Here, $f(Q)$ denotes $f$ evaluated at the first column of $Q$ viewed as a point in $\mathbb{P}^2(\mathbb{Z}/p\mathbb{Z})$.
\end{theorem}

This theorem plays a central role in computing the action of Hecke operators on the cuspidal cohomology.  For any $A \in \textsf{GL}(3, \mathbb{Q})$, there is a Hecke operator
\[T_A \colon H^3(\Gamma, \mathbb{C}) \to H^3(\Gamma, \mathbb{C}).\]
Its adjoint operator $T_A^*$ acts on the dual space $H_3(\Gamma, \mathbb{C})$.  Hence, we can evaluate the action of the Hecke operator through the intersection pairing.  More precisely, from \cite{AGG84}*{\S\ 4, p.\ 426}, we have the equation
\[\langle [Q], T_A^* f \rangle = \sum \langle [Q_{i,j}], f \rangle,\]
where the $[Q_{i,j}]$ are unimodular symbols such that $[QB_i] = \sum [Q_{i,j}]$ for a fixed $B_i$ in the finite collection of single coset representatives that stem from the decomposition of the double coset $\Gamma A \Gamma$:
\[\Gamma A \Gamma = \coprod_{i=1}^k B_i \Gamma, \quad B_i \in \textsf{GL}(3, \mathbb{Q}).\]
\begin{remark}
    We note that in \cite{AGG84}*{\S\ 4, p.\ 425}, the double coset $\Gamma A \Gamma$ should be decomposed as a finite union of left cosets, as opposed to right cosets.
\end{remark}
We will focus our attention on the special Hecke operators $E_\ell = T_{A(\ell)}$ and $F_\ell = T_{B(\ell)}$, with $\ell \neq p$ prime and
\[
A(\ell) =
\begin{pmatrix}
    \ell & 0 & 0 \\
    0 & 1 & 0 \\
    0 & 0 & 1
\end{pmatrix}, \quad
B(\ell) =
\begin{pmatrix}
    \ell & 0 & 0 \\
    0 & \ell & 0 \\
    0 & 0 & 1
\end{pmatrix},
\]
as these specific operators generate the Hecke algebra acting on $H^3(\Gamma, \mathbb{C})$ by \cite{AGG84}*{Prop.\ 4.1}.

\section{Main Results} \label{section:results}

Ideally, one works directly on $H^3_{\text{cusp}}(\Gamma, \mathbb{C})$, as a nonzero class in the cuspidal cohomology corresponds with a cuspidal automorphic form on $\GL(3)$.  Towards this goal, Ash, Grayson, and Green \cite{AGG84}*{Lem.\ 3.5} describe an isomorphism between the algebraic dual of $H^3_{\text{cusp}}(\Gamma, \mathbb{C})$ and $H_3(\Gamma, \mathbb{C}) / H_3^{\text{nc}}(\Gamma, \mathbb{C})$ induced by the natural pairing between the homology and cohomology. They then translate the problem into the vector space framework described in \cref{sec:vectorspace} to work with the spaces computationally.

\cref{thm:summary} describes a complement to $W^\text{nc}$ in $W$, which we denote $U$. By \cite{AGG84}*{Cor.\ 3.21}, we know that the dimension of $U$ is equal to the dimension of $H^3_{\text{cusp}}(\Gamma, \mathbb{C})$. Moreover, one can show that $U$ and $W^\text{nc}$ are disjoint.  Hence, $W$ decomposes as $W = U \oplus W^\text{nc}.$
Thus, elements $f \in U$ give a collection of coset representatives for $W/W^\text{nc}$.  Since $W^\text{nc}$ is Hecke invariant by \cite{AGG84}*{Rem.\ 3.22}, the method for computing Hecke operators on $W$ described in \cref{sec:mod_symbols} can be applied to computation on the quotient space $W/W^\text{nc}$, with minor adjustments.

In particular, the authors explain in \cite{AGG84}*{\S\ 6} that in order to compute the Hecke operators on the quotient space $W/W^\text{nc}$, one would need an appropriate collection of elements in the dual of $W$ that span the annihilator of $W^\text{nc}$.  They go on to note that finding a collection of individual unimodular symbols $[Q]$ that spans this annihilator is not possible, and consequently work with an alternate decomposition of $W$.  They show that $W = W' \oplus \im(\beta)$,
where $W'$ is the subspace of $W$ consisting of $f$'s which satisfy conditions (i)-(iii) of \cref{thm:summary} and $\im(\beta)$ is as defined in \cref{def:2}.

\begin{remark}
    We will use $\im(\beta)$ to denote the subspace of $W$ referred to as $\frac{1}{2}W^{\text{nc}}$ in \cite{AGG84}.  We opt for this notation to emphasize the role of the maps $\alpha$ and $\beta$ in our construction.
\end{remark}

We were able to construct a spanning set of operators in the dual of $W$ for the annihilator of $W^\text{nc}$, and hence, compute Hecke operators directly on the cuspidal cohomology.  The spanning set, presented in \cref{thm:A,thm:B} below, consists of linear combinations of the unimodular symbols from \cite{AGG84}*{\S\ 6, p.\ 431} that span the annihilator of $\im(\beta)$.  To motivate our construction, we briefly examine this collection of unimodular symbols, which contains unimodular symbols of the form:
\[
    [Q_{x,y}] = \begin{pmatrix}
        1 & 0 & 0 \\
        x & 1 & 0 \\
        y & 0 & 1
    \end{pmatrix},
    \quad x, y \in \mathbb{Z}/p\mathbb{Z}, \ xy \neq 0 \pmod p.
\]
We observe that such $[Q_{x, y}] \in \textsf{Ann}\left(\im(\beta)\right)$ because $\beta$ is supported on points $(x \colon y \colon z)$ with $xyz = 0$.  We leverage this idea instead for the annihilator of $W^\text{nc} = \mathsf{im}(\delta) = \mathsf{im}(\alpha) + \mathsf{im}(\beta)$ (see \cref{thm:Wnc}).

\begin{definition}\label{def:Rxyz}
    For $xyz \neq 0 \pmod p$, let $R_{x, y, z}$ denote the operator given by the following linear combination of unimodular symbols:
    \[
        R_{x, y, z} = [Q_{x,y}] + [Q_{y,z}] + [Q_{z,x}] - [Q_{\frac{y}{x},\frac{z}{x}}].
    \]
\end{definition}

\begin{theorem}\label{thm:A}
    With $R_{x,y,z}$ as above, $R_{x,y,z} \in \textsf{Ann}(W^{\text{nc}})$.
\end{theorem}

\begin{proof}
      We start by noting that an arbitrary $R_{x,y,z}$ is constructed as a sum of modular symbols that vanish on $\mathsf{im}(\beta)$; hence, it too vanishes on $\mathsf{im}(\beta)$.  Further, for $f = \alpha f' \in \mathsf{im}(\alpha)$, we invoke \cref{def:2} and observe that
    \begin{align*}
        (\alpha f')(R_{x,y,z}) & = (\alpha f')(Q_{x,y}) + (\alpha f')(Q_{y,z}) + (\alpha f')(Q_{z,x}) - (\alpha f')(Q_{\frac{y}{x},\frac{z}{x}}) \\
        & = \alpha f'(1 \colon x \colon y) + \alpha f'(1 \colon y \colon z) + \alpha f'(1 \colon z \colon x) - \alpha f'\left(1 \colon \frac{y}{x} \colon \frac{z}{x}\right) \\
        & =  \alpha f'(1 \colon x \colon y) + \alpha f'(1 \colon y \colon z) + \alpha f'(1 \colon z \colon x) - \alpha f'(x \colon y \colon z) \\
        & = f'(1 \colon x) + f'(x \colon y) + f'(y \colon 1) + f'(1 \colon y) + f'(y \colon z) + f'(z \colon 1) \\ & \hspace{2em} + f'(1 \colon z) + f'(z \colon x) + f'(x \colon 1) - (f'(x \colon y) + f'(y \colon z) + f'(z \colon x)) \\
        & = ( f'(1 \colon x) + f'(x \colon 1) ) + ( f'(1 \colon y) + f'(y \colon 1) ) + ( f'(1 \colon z) + f'(z \colon 1) ) \\ & \hspace{2em} + (f'(x \colon y) + f'(y \colon z) + f'(z \colon x)) - (f'(x \colon y) + f'(y \colon z) + f'(z \colon x)) \\
        & = 0
    \end{align*}
    The cancellations in the final step stem from the properties that $f' \in W_0(\Delta(p))$ enjoys from \cref{def:1}.
\end{proof}

We now turn our attention to showing that the collection $\{R_{x,y,z}\}$ spans the annihilator of $W^{\text{nc}}$ in the dual of $W$.  Let $V$ denote the span of this collection.  We will show that $V = \textsf{Ann}(W^\text{nc})$ by proving the equivalent statement $\textsf{Ann}(V) = W^{\text{nc}}$, where $\textsf{Ann}(V)$ denotes the annihilator of $V$ in $W$.  With this goal in mind, we make the following definitions akin to \cref{def:3}.

\begin{definition}
    Let $C, D \colon \textsf{Ann}(V) \to W_0(\Delta(p))$ be the maps such that
    \begin{align*}
        (Cf)(x \colon y) &= \frac{\sum_{1 \leq \lambda \leq p-1} f(1 \colon \lambda x \colon \lambda y)}{p-1}; \\
        (Df)(x \colon y) &= -(Cf)(x \colon y) + (Bf)(x \colon y).
    \end{align*}
    Here, $B$ is the linear map given in \cref{def:3}.
\end{definition}

We note that these maps are well-defined.  Moreover, since the image of $B$ lies in $W_0(\Delta(p))$, we see that the image of $D$ lies in $W_0(\Delta(p))$ so long as the image of $C$ does.  We leave it to the reader to verify this claim by confirming that all the conditions of \cref{def:1} hold for $Cf$.  

\begin{theorem}\label{thm:B}
    Let $V$ be the span of $\{R_{x,y,z}\}$ with $R_{x,y,z}$ as in \cref{def:Rxyz}.  Then, $V = \textsf{Ann}(W^{\text{nc}})$.
\end{theorem}

\begin{proof}
    As previously stated, we will show that $V = \textsf{Ann}(W^{\text{nc}})$ by proving the equivalent statement $\textsf{Ann}(V) = W^{\text{nc}}$.  By \cref{thm:A}, $W^{\text{nc}} \subseteq \textsf{Ann}(V)$.  To prove the reverse containment, take $f \in \textsf{Ann}(V)$.  To show $f \in W^\text{nc} \simeq \mathsf{im}(\delta)$, we must show that $f = f_1 + f_2$, where $f_1 \in \mathsf{im}(\alpha)$ and $f_2 \in \mathsf{im}(\beta)$.

    \begin{samepage}
    If such an $f_1$ and $f_2$ exist, then, by \cref{def:2}, for $(x \colon y \colon z) \in \mathbb{P}^2(\mathbb{Z}/p\mathbb{Z})$ with $xyz \neq 0 \pmod p$, $f_2$ vanishes and $f$ is determined completely by $f_1$. Thus, we can define $f_1$ by investigating $f$ evaluated at these specific points.  Let $xyz \neq 0 \pmod p$.  Then, we note the following:
    \begin{align*}
        f(x \colon y \colon z) & = \frac{(p-1)(f(x \colon y \colon z))}{p-1} \\
        & = \frac{\sum_{1 \leq \lambda \leq p-1} f(\lambda x \colon \lambda y \colon \lambda z)}{p-1} \\
        & = \frac{\sum_{1 \leq \lambda \leq p-1} f(1 \colon \lambda x \colon \lambda y)}{p-1} + \frac{\sum_{1 \leq \lambda \leq p-1} f(1 \colon \lambda y \colon \lambda z)}{p-1} \\
        & \hspace{2em} + \frac{\sum_{1 \leq \lambda \leq p-1} f(1 \colon \lambda z \colon \lambda x)}{p-1} \\
        & = (Cf)(x \colon y) + (Cf)(y \colon z) + (Cf)(z \colon x) \\
        & = (\alpha(Cf))(x \colon y \colon z)
    \end{align*}
    \end{samepage}
    The expansion in the third step is possible since $f \in \textsf{Ann}(V)$ and therefore,
    \[f(R_{x,y,z}) = f(Q_{x,y}) + f(Q_{y, z}) + f(Q_{z,x}) - f(Q_{\frac{y}{x},\frac{z}{x}}) = 0,\]
    which implies
    \[f(x \colon y \colon z) = f(1 \colon x \colon y) + f(1 \colon y \colon z) + f(1 \colon z \colon x).\]
    So, we will set $f_1 = \alpha(Cf)$.

    To define $f_2$, we now check the case when $xyz = 0 \pmod p$.  Without loss of generality, we let $z = 0$ because if any other coordinate is zero, the evaluation of $f$ at that point is equal to $f(x \colon y \colon 0)$ by property (i) of \cref{thm:W}. Therefore,
    \begin{align*}
        f(x \colon y \colon 0) & = (Bf)(x \colon y) \\
        & = (Cf)(x \colon y) - (Cf)(x \colon y) + (Bf)(x \colon y) \\
        & = (Cf)(x \colon y) + (Df)(x \colon y) \\
        & = (\alpha(Cf))(x \colon y \colon 0) + (\beta(Df))(x \colon y \colon 0).
    \end{align*}
    Set $f_2 = \beta(Df)$.  Hence, we have found $f_1 \in \mathsf{im}(\alpha)$ and $f_2 \in \mathsf{im}(\beta)$, as desired.
\end{proof}

\section{Computational Method} \label{sec:comp}

The code for the computations described in the subsequent section has implementations in both \textsc{SageMath} \cite{Sage} and \textsc{Magma} \cite{Magma}.  The scripts and additional data can be found in the GitHub repository \cite{Code}. The computations were performed on Wesleyan University's High Performance Compute Cluster.

\subsection{Computation of Dimension of Cuspidal Cohomology} \label{sec:comp1} To compute the dimension of $H^3_{\text{cusp}}(\Gamma, \mathbb{C})$, we first constructed a matrix $M$ whose kernel is the vector space $U$.  In order to build $M$, we note that one can view a function $f \in U$ as the vector of its values, where components are indexed by points in $\mathbb{P}^2(\mathbb{Z}/p\mathbb{Z})$.  Thus, ranging over points in $\mathbb{P}^2(\mathbb{Z}/p\mathbb{Z})$, we can construct rows of this matrix using conditions (i)-(iv) from \cref{thm:summary}.

For computational efficiency, $M$ was built over the finite field $\mathbb{F}_q$ for the prime $q = 12379$.  The rank was then computed using either the built-in \texttt{LinBox} implementation in \textsc{SageMath} or the generic \texttt{Rank} function in \textsc{Magma}, depending on the chosen script.  From the rank calculation, we could next find the dimension of the kernel, i.e.\ the dimension of $U$, which is equal to the dimension of $H^3_{\text{cusp}}(\Gamma, \mathbb{F}_q)$ by \cite{AGG84}*{Cor.\ 3.21}.  We computed the dimension of $H^3_{\text{cusp}}(\Gamma, \mathbb{F}_q)$ for all primes $p$ less than 3500 using this method.

If $\mathsf{dim}(H^3_{\text{cusp}}(\Gamma, \mathbb{F}_q))$ was zero following the initial rank calculation, we could conclude immediately that $\mathsf{dim}(H^3_{\text{cusp}}(\Gamma, \mathbb{C}))$ was also zero.  If $\mathsf{dim}(H^3_{\text{cusp}}(\Gamma, \mathbb{F}_q))$ was nonzero, we then rebuilt $M$ over $\mathbb{F}_q$ for $q = 31991$.  We calculated the dimension over this second finite field to further support our findings.  Besides the values $p \in \{53, 61, 79, 89, 223\}$ previously found in \cite{AGG84} and \cite{vGvdKTV97}, $\mathsf{dim}(H^3_{\text{cusp}}(\Gamma, \mathbb{F}_q))$ was nonzero only for $p \in \{521, 953, 1289, 1433, 1913, 2089, 2833\}$.  For all $p$ listed above, $\mathsf{dim}(H^3_{\text{cusp}}(\Gamma, \mathbb{F}_q)) = 2$ for both $q = 12379$ and $q = 31991$.

We opted to compute $\mathsf{dim}(H^3_{\text{cusp}}(\Gamma, \mathbb{F}_q))$ for two main reasons.  First, we wanted to leverage the high-efficiency algorithms from \texttt{LinBox}, which are only implemented over the integers and finite fields.  Second, we wished to avoid any floating-point error that would arise had we computed over $\mathbb{C}$ directly.  In \cite{AGM10}*{\S\ 2, p.\ 1814}, the authors explain that using this approach is reasonable because for large enough $q$, we expect with high likelihood that $\mathsf{dim}(H^3_{\text{cusp}}(\Gamma, \mathbb{F}_q))$ and $\mathsf{dim}(H^3_{\text{cusp}}(\Gamma, \mathbb{C}))$ are equal.  Moreover, the Hecke computations work in a compatible way.

In order to compute the Hecke operators, as described in the next section, we needed an explicit basis for $W/W^\text{nc}$.  By construction, we could use the basis for the kernel of $M$.  For $p \in \{ 521, 953, 1289, 1433, 1913\}$, we could compute this basis; however, for $p = 2089,\ 2833$, the matrices were too large for effective computation. As a result, we were unable to produce bases for these levels.

\subsection{Computation of Hecke Operators} With the bases for $W/W^\text{nc}$ in hand, we now turn our attention to the computation of the Hecke operators $E_\ell$ and $F_\ell$, as defined at the end of $\cref{sec:mod_symbols}$.  Recall, these specific operators generate the entire Hecke algebra acting on $H^3(\Gamma, \mathbb{C})$.  In order to compute the action of $E_\ell$ and $F_\ell$, we need a choice of single coset representatives $B_i$ for $A(\ell)$ and $B(\ell)$ respectively; we used the same representatives found in \cite{AGG84}*{\S\ 6, pp.\ 430-431}.

We then implemented the algorithm described in \cite{vGvdKTV97}*{\S\ 2.10} to reduce the linear combination
\[R_{x, y, z}B_i = [Q_{x,y}B_i] + [Q_{y,z}B_i] + [Q_{z,x}B_i] - [Q_{\frac{y}{x},\frac{z}{x}}B_i] \]
into a sum of unimodular symbols by reducing each modular symbol in the linear combination.  Having shown that the collection $\{R_{x, y, z}\}$ spans the annihilator of $W^{\text{nc}}$ in \cref{thm:B}, we can directly compute the action of Hecke operators on $W/W^\text{nc}$, which corresponds to $H^3_{\text{cusp}}(\Gamma, \mathbb{F}_q)$.

By comparison, in \cite{AGG84}*{\S\ 6, p.\ 432}, the authors worked on $W/\im(\beta)$ and therefore needed to filter out additional, non-cuspidal eigenvalues on $W^{\text{nc}} / \im(\beta)$ for $E_\ell$ and $F_\ell$.  This filtering process becomes more complicated as the dimension of $W / \im(\beta)$ grows because the number of non-cuspidal eigenvalues increases along with the level, as seen in \cref{table:1}.  In contrast, our method does not pick up any of these extra eigenvalues, and immediately gives the desired cuspidal data.

\begin{table}
    \small
    \caption{Prime level for all $p < 3500$ at which nonzero cuspidal classes appear and the corresponding dimensions of $W / \im(\beta)$ and $W / W^{\text{nc}}$.}
    \label{table:1}
    \arraycolsep=1.7em
    \def\arraystretch{1.2}
    $\begin{array}{c c c}
        \toprule
        p & \textsf{dim}(W / \im(\beta)) & \textsf{dim}(W / W^{\text{nc}}) \\
        \midrule
        53 & 6 & 2 \\
        61 & 6 & 2 \\
        79 & 8 & 2 \\
        89 & 9 & 2 \\
        223 & 20 & 2 \\
        521 & 45 & 2 \\
        953 & 81 & 2 \\
        1289 & 109 & 2 \\
        1433 & 121 & 2 \\
        1913 & 161 & 2 \\
        2089 & 175 & 2 \\
        2833 & 237 & 2 \\
        \bottomrule
    \end{array}$
\end{table}

Since we are working directly on $W/W^\text{nc}$, our computation gives the characteristic polynomial of $E_\ell$ (resp.\ $F_\ell$) on $H^3_{\text{cusp}}(\Gamma, \mathbb{F}_q)$.  However, recall that we want to understand the action of these operators on $H^3_{\text{cusp}}(\Gamma, \mathbb{C})$.  A consequence of working over $\mathbb{F}_q$ is that the characteristic polynomial returned by our calculations is the reduction modulo $q$ of the characteristic polynomial of $E_\ell$ (resp.\ $F_\ell$) on $H^3_{\text{cusp}}(\Gamma, \mathbb{C})$.  We let $\varphi_\ell$ denote the characteristic polynomial on $H^3_{\text{cusp}}(\Gamma, \mathbb{C})$ and $\varphi_\ell^{\mathrm{red}}$ denote its reduction modulo $q$.  Our goal is to find $\varphi_\ell$, as its roots, $e_\ell$ and complex conjugate $\overline{e_\ell}$ (resp.\ $f_\ell$ and $\overline{f_\ell}$), are central in computing local factors of the $L$-function of the associated cuspidal automorphic form.

To recover $\varphi_\ell$ from $\varphi_\ell^{\mathrm{red}}$, we leverage Ramanujan's conjecture, which bounds the absolute value of $e_\ell$ (resp. $f_\ell$) by $3\ell$.  In turn, this gives bounds on the coefficients of $\varphi_\ell$, which determine $\varphi_\ell$ explicitly when $\ell < \frac{\sqrt{q}}{3}$.  We can then use these $\varphi_\ell$ for small $\ell$ to determine the splitting field generated by the eigenvalues $e_\ell$ (resp.\ $f_\ell$), which we know to be either totally real or CM by \cite{APT91}*{Lem.\ 1.3}.  With this knowledge of the splitting field and the bounds given by Ramanujan's conjecture, we can find $\varphi_\ell$ for larger $\ell$.

For the five levels on which we could perform computations of Hecke operators, i.e.\ the levels from \cref{sec:comp1} whose bases we were able to compute, \cref{table:2} reports the eigenvalues $e_\ell$ of $E_\ell$ for a fixed common eigenvector, the corresponding characteristic polynomials $\varphi_\ell$ of $E_\ell$ on $H^3_{\text{cusp}}(\Gamma, \mathbb{C})$, and the field generated by the eigenvalues.  We note that in order to determine which eigenvalue coincides with the fixed common eigenvector, $\varphi_\ell^{\mathrm{red}}$ cannot be irreducible over $\mathbb{F}_q$.  To ensure this condition was met for all $\ell$ tested, $q = 12379$ was used for levels $521$ and $953$, and $q = 13001$ was used for the remaining levels.  Both choices of $q$ allowed us to determine $\varphi_\ell$ explicitly for $\ell \leq 37$.  We include additional calculations for $37 < \ell < 50$ in \cref{table:2} to illustrate the strategy of finding $\varphi_\ell$ from $\varphi_\ell^{\mathrm{red}}$ for $\ell \geq \frac{\sqrt{q}}{3}$.  (The data available in \cite{Code} contain further results for $50 < \ell < 150$.)

\begin{table}
    \small
    \caption{Hecke eigenvalues $e_\ell$ for Hecke operators $E_\ell$ for a fixed common eigenvector and corresponding characteristic polynomials $\varphi_\ell$ on $H^3_{\text{cusp}}(\Gamma, \mathbb{C})$.}
    \label{table:2}
    \arraycolsep=0.9em
    \def\arraystretch{1.2}
    $\begin{array}{ r | r l | r l }
        \toprule
        & \multicolumn{2}{c |}{p = 521} & \multicolumn{2}{c}{p = 953} \\
        & k \simeq \mathbb{Q}(\sqrt{-2}) \hfil &  \hfil \omega = \sqrt{-2} & k \simeq \mathbb{Q}(\sqrt{-2}) \hfil & \hfil \omega = \sqrt{-2} \\
        \midrule
        \ell & e_\ell \hfil & \hfil \varphi_\ell & e_\ell \hfil & \hfil \varphi_\ell \\
        \midrule
        2 & 1 & (T - 1)^2 & 1 & (T - 1)^2 \\
        3 & \omega - 1 & T^2 + 2T + 3 & -\omega - 1 & T^2 + 2T + 3 \\
        5 & 4\omega - 5 & T^2 + 10T + 57 & 1 & (T - 1)^2 \\
        7 & 3\omega  - 3 & T^2 + 6T + 27 & 1 & (T - 1)^2 \\
        11 & 2\omega  + 1 & T^2 - 2T + 9 & -3\omega + 7 & T^2 - 14T + 67 \\
        13 & -6\omega  + 1 & T^2 - 2T + 73 & -12\omega - 9 & T^2 + 18T + 369 \\
        17 & 12\omega  - 17 & T^2 + 34T + 577 & 8\omega + 1 & T^2 - 2T + 129 \\
        19 & -3\omega + 7 & T^2 - 14T + 67 & -3\omega + 11 & T^2 - 22T + 139 \\
        23 & -9\omega + 13 & T^2 - 26T + 331 & -13\omega - 11 & T^2 + 22T + 459 \\
        29 & -18\omega + 1 & T^2 - 2T + 649 & 8\omega + 7 & T^2 - 14T + 177 \\
        31 & -12\omega - 7 & T^2 + 14T + 337 & 33\omega - 23 & T^2 + 46T + 2707 \\
        37 & 18\omega + 9 & T^2 - 18T + 729 & -12\omega - 9 & T^2 + 18T + 369 \\
        41 & 32\omega - 29 & T^2 + 58T + 2889 & -12\omega + 37 & T^2 - 74T + 1657 \\
        43 & -9\omega - 33 & T^2 + 66T + 1251 & 15\omega + 3 & T^2 - 6T + 459 \\
        47 & -2\omega+  73 & T^2 - 146T + 5337 & 45\omega - 47 & T^2 + 94T + 6259 \\
        \bottomrule
    \end{array}$
\end{table}

\begin{table}
    \small
    \arraycolsep=0.9em
    \def\arraystretch{1.2}
    $\begin{array}{ r | r l | r l }
        \toprule
        & \multicolumn{2}{c |}{p = 1289} & \multicolumn{2}{c}{p = 1433} \\
        & k \simeq \mathbb{Q}(\sqrt{-1}) \hfil &  \hfil i = \sqrt{-1} & k \simeq \mathbb{Q}(\sqrt{-1}) \hfil & \hfil i = \sqrt{-1} \\
        \midrule
        \ell & e_\ell \hfil & \hfil \varphi_\ell & e_\ell \hfil & \hfil \varphi_\ell \\
        \midrule
        2 & 2i - 1 & T^2 + 2T + 5 & 2i - 1 & T^2 + 2T + 5 \\
        3 & 2i + 1 & T^2 - 2T + 5 & 2i + 1 & T^2 - 2T + 5 \\
        5 & 2i + 2 & T^2 - 4T + 8 & -4i - 1 & T^2 + 2T + 17 \\
        7 & -2i + 5 & T^2 - 10T + 29 & 5i - 3 & T^2 + 6T + 34 \\
        11 & -5i - 5 & T^2 + 10T + 50 & 10i + 1 & T^2 - 2T + 101 \\
        13 & 4i - 11 & T^2 + 22T + 137 & 2i + 16 & T^2 - 32T + 260 \\
        17 & -4i + 23 & T^2 - 46T + 545 & -24i - 17 & T^2 + 34T + 865 \\
        19 & 4i + 1 & T^2 - 2T + 17 & -3i + 11 & T^2 - 22T + 130 \\
        23 & -18i + 1 & T^2 - 2T + 325 & 17i - 19 & T^2 + 38T + 650 \\
        29 & 34i - 4 & T^2 + 8T + 1172 & 8i - 19 & T^2 + 38T + 425 \\
        31 & -20i + 25 & T^2 - 50T + 1025 & 20i - 11 & T^2 + 22T + 521 \\
        37 & 24i + 31 & T^2 - 62T + 1537 & -30i + 2 & T^2 - 4T + 904 \\
        41 & -20i - 40 & T^2 + 80T + 2000 & -20i + 32 & T^2 - 64T + 1424 \\
        43 & -39i + 7 & T^2 - 14T + 1570 & 22i - 39 & T^2 + 78T + 2005 \\
        47 & -29i - 7 & T^2 + 14T + 890 & 77i - 31 & T^2 + 62T + 6890 \\
        \bottomrule
    \end{array}$
\end{table}

\begin{table}
    \small
    \caption*{\textsc{Table 2 Continued}}
    \arraycolsep=0.9em
    \def\arraystretch{1.2}
    $\begin{array}{ r | r l }
        \toprule
        & \multicolumn{2}{c}{p = 1913} \\
        & k \simeq \mathbb{Q}(\sqrt{-1}) \hfil &  \hfil i = \sqrt{-1} \\
        \midrule
        \ell & e_\ell \hfil & \hfil \varphi_\ell \\
        \midrule
        2 & 2i - 1 & T^2 + 2T + 5 \\
        3 & -i - 1 & T^2 + 2T + 2 \\
        5 & -4i - 1 & T^2 + 2T + 17 \\
        7 & 4i + 1 & T^2 - 2T + 17 \\
        11 & -5i + 3 & T^2 - 6T + 34 \\
        13 & -16i - 11 & T^2 + 22T + 377 \\
        17 & 12 & (T - 12)^2 \\
        19 & -21i - 9 & T^2 + 18T + 522 \\
        23 & -14i - 3 & T^2 + 6T + 205 \\
        29 & 22i + 6 & T^2 - 12T + 520 \\
        31 & -10i + 1 & T^2 - 2T + 101 \\
        37 & -16i - 29 & T^2 + 58T + 1097 \\
        41 & 85 & (T - 85)^2 \\
        43 & -32i + 5 & T^2 - 10T + 1049 \\
        47 & -26i + 1 & T^2 - 2T + 677 \\
        \bottomrule
    \end{array}$
\end{table}

Now, let $k$ denote the splitting field generated by the eigenvalues.  At levels 521 and 953, $k \simeq \mathbb{Q}(\sqrt{-2})$, and at levels 1289, 1433, and 1913, $k \simeq \mathbb{Q}(\sqrt{-1})$.  We make the following observations in line with those of \cite{AGG84}*{\S\ 6, p.\ 434}.  First, we find that the ring spanned by the images of Hecke operators in $k$ is $\mathbb{Z} + \mathbb{Z}\sqrt{D}$, where $D$ denotes the negative square free rational integer such that $k \simeq \mathbb{Q}(\sqrt{D})$.  Second, we note that both in $\mathbb{Q}(\sqrt{-1})$ and in $\mathbb{Q}(\sqrt{-2})$,  the only prime that ramifies is $2$.  Then, as predicted, we see that the level $p$ is a square in $\mathbb{Z}_2$ for each $p \in \{521, 953, 1289, 1433, 1913\}$.

We conclude by noting that the dimension of the cuspidal cohomology is two for all levels at which nonzero cuspidal classes appear.  We know of no explanation as to why this dimension does not increase with $p$.  There is no guarantee that the dimension remains constant for higher level.

\section*{Acknowledgments}
The author thanks David Pollack for helpful discussions and insights during the preparation of the article.  The author also thanks Wesleyan University for computer time supported by the NSF under grant numbers CNS-0619508 and CNS-0959856.  Finally, the author thanks the referees for their constructive feedback.


\begin{bibdiv}
    \begin{biblist}

    \bib{AGG84}{article}{
          author={Ash, Avner},
          author={Grayson, Daniel},
          author={Green, Philip},
           title={Computations of cuspidal cohomology of congruence subgroups of {${\rm SL}(3,{\bf Z})$}},
            date={1984},
            ISSN={0022-314X,1096-1658},
         journal={J. Number Theory},
          volume={19},
          number={3},
           pages={412\ndash 436},
             url={https://doi.org/10.1016/0022-314X(84)90081-7},
          review={\MR{769792}},
    }

    \bib{AGM10}{article}{
          author={Ash, Avner},
          author={Gunnells, Paul~E.},
          author={McConnell, Mark},
           title={Cohomology of congruence subgroups of {${\rm SL}_4(\bf Z)$}. {III}},
            date={2010},
            ISSN={0025-5718,1088-6842},
         journal={Math. Comp.},
          volume={79},
          number={271},
           pages={1811\ndash 1831},
             url={https://doi.org/10.1090/S0025-5718-10-02331-8},
          review={\MR{2630015}},
    }

    \bib{AM92}{article}{
          author={Ash, Avner},
          author={McConnell, Mark},
           title={Experimental indications of three-dimensional {G}alois representations from the cohomology of {${\rm SL}(3,{\bf Z})$}},
            date={1992},
            ISSN={1058-6458,1944-950X},
         journal={Experiment. Math.},
          volume={1},
          number={3},
           pages={209\ndash 223},
             url={http://projecteuclid.org/euclid.em/1048622024},
          review={\MR{1203875}},
    }

    \bib{APT91}{article}{
       author={Ash, Avner},
       author={Pinch, Richard},
       author={Taylor, Richard},
       title={An $\widehat{A_4}$ extension of ${\bf Q}$ attached to a
       nonselfdual automorphic form on ${\rm GL}(3)$},
       journal={Math. Ann.},
       volume={291},
       date={1991},
       number={4},
       pages={753--766},
       issn={0025-5831},
       review={\MR{1135542}},
    }

    \bib{AR79}{article}{
          author={Ash, Avner},
          author={Rudolph, Lee},
           title={The modular symbol and continued fractions in higher dimensions},
            date={1979},
            ISSN={0020-9910,1432-1297},
         journal={Invent. Math.},
          volume={55},
          number={3},
           pages={241\ndash 250},
             url={https://doi.org/10.1007/BF01406842},
          review={\MR{553998}},
    }

    \bibitem[Magma]{Magma}
    W. Bosma, J. J. Cannon, C. Fieker, A. Steel (eds.), \textit{Handbook of Magma functions}, Version 2.28-14 (2024). \href{https://magma.maths.usyd.edu.au/magma/handbook/}{\tt https://magma.maths.usyd.edu.au/magma/handbook/}.

    \bib{LS82}{article}{
          author={Lee, Ronnie},
          author={Schwermer, Joachim},
           title={Cohomology of arithmetic subgroups of {${\rm SL}\sb{3}$} at infinity},
            date={1982},
            ISSN={0075-4102,1435-5345},
         journal={J. Reine Angew. Math.},
          volume={330},
           pages={100\ndash 131},
             url={https://doi.org/10.1515/crll.1982.330.100},
          review={\MR{641814}},
    }

    \bib{Code}{misc}{
          author={Porat, Zachary},
           title={GitHub repository for related scripts and data},
            note={Available at \href{https://github.com/zporat/cuspidal-cohomology}{\texttt{https://github.com/zporat/cuspidal-cohomology}} (updated 2025)},
            label={Por25}
    }

    \bib{Sage}{misc}{
          author={{The Sage Developers}},
           title={{S}age{M}ath, the {S}age {M}athematics {S}oftware {S}ystem},
            date={Version 10.4 (2024)},
            note={\href{https://www.sagemath.org}{\tt https://www.sagemath.org}},
            label={Sage}
    }

    \bib{vGT94}{article}{
          author={van Geemen, Bert},
          author={Top, Jaap},
           title={A non-selfdual automorphic representation of {${\rm GL}_3$} and a {G}alois representation},
            date={1994},
            ISSN={0020-9910,1432-1297},
         journal={Invent. Math.},
          volume={117},
          number={3},
           pages={391\ndash 401},
             url={https://doi.org/10.1007/BF01232250},
          review={\MR{1283724}},
    }

    \bib{vGvdKTV97}{article}{
          author={van Geemen, Bert},
          author={van~der Kallen, Wilberd},
          author={Top, Jaap},
          author={Verberkmoes, Alain},
           title={Hecke eigenforms in the cohomology of congruence subgroups of {${\rm SL}(3,{\bf Z})$}},
            date={1997},
            ISSN={1058-6458,1944-950X},
         journal={Experiment. Math.},
          volume={6},
          number={2},
           pages={163\ndash 174},
             url={http://projecteuclid.org/euclid.em/1047650002},
          review={\MR{1474576}},
    }
    \end{biblist}
\end{bibdiv}
\end{document}